\documentclass[reqno,12pt]{amsart}
\usepackage{ragged2e}
\usepackage{amssymb}
\usepackage{amsfonts}
\usepackage{amsmath}
\usepackage{amsthm}
\usepackage{mathrsfs}
\usepackage{diagbox}
\usepackage{color,soul}
\usepackage{bbm}
\usepackage{amsthm, graphicx,empheq}
\usepackage{framed}
\usepackage{bm}
\usepackage{tikz}
\usepackage{subfigure}
\usepackage{float}
\usepackage[backref,colorlinks,linkcolor=red,anchorcolor=green,citecolor=blue]{hyperref}
\usepackage[normalem]{ulem}
\usepackage{xcolor}

\definecolor{highlightcolor}{RGB}{255,255,0}

\sethlcolor{yellow} 

\newtheorem{lemma}{Lemma}[section]
\newtheorem{theorem}{Theorem}[section]
\newtheorem{problem}{Problem}[section]

\newtheorem{remark}{Remark}[section]
\numberwithin{equation}{section}

\makeatletter
\@namedef{subjclassname@2020}{\textup{2020} Mathematics Subject Classification}
\makeatother

\begin{document}
	
	\title[Supersonic flow of past a conical body]{Global solutions for supersonic flow of a Chaplygin gas past a conical wing with a shock wave detached from the leading edges}

    \author{Bingsong Long}	

\address[Bingsong Long]{School of Mathematics and Statistics, Huanggang Normal University, Hubei 438000, China}\email{\tt longbingsong@hgnu.edu.cn}	

	\keywords{supersonic flow; the conical wing; Chaplygin gas; Euler equations; degenerate elliptic equation.}

	\subjclass[2020]{35L65, 35L67, 35J25, 35J70, 76N10}
		
	\date{\today}
	
	\begin{abstract}
In this paper, we first investigate the mathematical aspects of supersonic flow of a Chaplygin gas past a conical wing with diamond-shaped cross sections in the case of a shock wave detached from the leading edges. The flow under consideration is governed by the three-dimensional steady compressible Euler equations. For the Chaplygin gas, all characteristics are linearly degenerate, and shocks are reversible and characteristic. Using these properties, we can determine the location of the shock in advance and reformulate our problem as an oblique derivative problem for a nonlinear degenerate elliptic equation in conical coordinates. By establishing a Lipschitz estimate, we show that the equation is uniformly elliptic in any subdomain strictly away from the degenerate boundary, and then further prove the existence of a solution to the problem via the continuity method and vanishing viscosity method.
	\end{abstract}
	
	\allowbreak
	\allowdisplaybreaks	
	\maketitle
	
	
\section{Introduction}\label{sec:1}
The problem of supersonic flow past a conical body is of great importance in mathematical gas dynamics because of its wide applications in aeronautics \cite{Anderson03}. When a conical body is placed in supersonic flow, a shock front is produced around its head. Such a problem has been extensively investigated in physical experiments and numerical simulations \cite{Babenko64,Chushkin70}. However, all efforts to mathematically analyze this problem have mainly focused on some special scenarios. Courant-Friedrichs \cite{Courant-Friedrichs85} proved that when the conical body is a circular cone and the supersonic flow hits this cone without angle of attack (i.e., the flow is parallel to the axis of the cone), a self-similar solution of this problem can be established. Using this self-similar solution as a background solution, Chen \cite{Chen03} showed the global existence of self-similar solutions for the case where the conical body is marginally different from a circular cone. Also, for the case of curved cones, some related problems have been well studied; see, for example, \cite{CF09,CKZ21,ChenXY00,CuiY09,HZ19,LWY14,LienY99,QY20,WangZ09}. When the conical body is a delta wing, which is not a small perturbation of a circular cone, Chen-Yi \cite{CY15} and Long-Yi \cite{LY22} studied the global existence of self-similar solutions under the assumption that the shock front is attached to the leading edges. The study of this scenario is important in both engineering and mathematics, as most hypersonic vehicles are designed as a delta wing \cite{AL85}. As mentioned in \cite[p.910]{Anderson17}, for conventional hypersonic vehicles, the shock front is usually detached from the leading edge, as sketched in Figure \ref{fig1}, which is reproduced from \cite[Figure 14.22b]{Anderson17}. 
Therefore, in this paper, as a continuation of \cite{CY15,LY22}, we further study the problem for the case in which the shock front is only attached to the apex but away from the leading edges in a Chaplygin gas. To the best of our knowledge, there are currently no relevant mathematical theoretical results.

\begin{figure}[htb]
\centering
\includegraphics[scale=0.7]{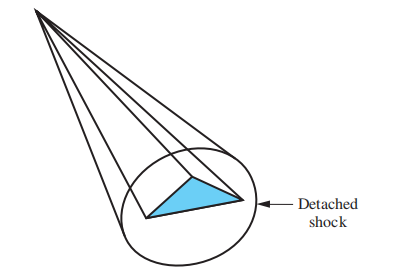}
\caption{The generic hypersonic configurations.}\label{fig1}
\end{figure} 

The state equation for the Chaplygin gas is governed by \cite{Chap04}:
    \begin{equation}\label{eq:Chaplygin gas}
	  p(\rho)=A\Big(\frac{1}{\rho_{*}}-\frac{1}{\rho}\Big),
    \end{equation}
where $p, \rho$ represent, respectively, the pressure and density; $\rho_{*}>0, A>0$ are two constants. As a model of dark energy in the universe, the Chaplygin gas appears in a number of cosmological theories \cite{KMP01,Popov10}; besides, the gas is also employed in the calculation of pressure distribution for subsonic airfoil design by Hsue-shen Tsien \cite{T1939}. Since the following proof heavily relies on the properties of the shock waves for the Chaplygin gas, let us briefly recall some basic properties of the Chaplygin gas. Detailed proofs of these properties can be found in \mbox{\cite[Section 2]{Serre09} and \cite[Appendix A.1]{Serre11}}. The basic facts for the solutions of three-dimensional steady compressible Euler equations for the Chaplygin gas are:

\mbox{$(1)$~}For a piecewise smooth steady flow, irrotationality and isentropy on one side persist on the other; namely, such flows are not only good approximations of Euler flows but genuine flows of the full gas with the Chaplygin equation of state.

\mbox{$(2)$~}All characteristics are linearly degenerate. In fact, it follows from \mbox{\eqref{eq:Chaplygin gas}} that the sound speed of the gas satisfies \mbox{$c=c(\rho)=\sqrt{A}/\rho$}, which implies \mbox{${\partial(\rho c)}/{\partial\rho}\equiv0$}. Moreover, any particle of the fluid moves across the steady shock at sonic speed.

The above property $(2)$ means that shocks are reversible and characteristic. We remark here that such properties also hold in the two-dimensional unsteady case, and the study of multi-dimensional Riemann problems has yielded satisfactory results \cite{CCW22,SL16,Serre09}. 

Now, let us describe the problem in detail. Denote by $\mathcal{C}_{\sigma_1}^{\sigma_2}$ a conical wing with diamond-shaped cross sections, in which the angle between the left and right edges is $2\sigma_1$ and the angle between the upper and lower edges is $2\sigma_2$ with $\sigma_1,\sigma_2\in(0,\pi/2)$. We place $\mathcal{C}_{\sigma_1}^{\sigma_2}$ in a $(x_1,x_2,x_3)$-coordinates such that it is symmetric with respect to the $x_{1}Ox_{3}$-plane and $x_{2}Ox_{3}$-plane, with its apex at the origin (see Figure \ref{fg1}), i.e.,
\begin{equation}\label{eq:conical body}
	\mathcal{C}_{\sigma_1}^{\sigma_2}:=\{(x_1,x_2,x_3):|x_{2}|<x_{3}\tan\sigma_1, |x_{1}|<(x_3-|x_2|\cot\sigma_1)\tan\sigma_2,x_{3}>0\}.
\end{equation}
The oncoming flow of state $(\rho_{\infty},\bm{v}_{\infty})$ is assumed to be uniform and supersonic, past the conical wing $\mathcal{C}_{\sigma_1}^{\sigma_2}$ without angle of attack, i.e., the velocity $\bm{v}_{\infty}=(0,0,v_{3\infty})$. Write
\begin{align*}
	\mathcal{D}&:=\{f(x_{2},x_{3})<x_{1}<0,x_{2}>0,x_{3}>0\},
\end{align*}
where $x_1=f(x_{2},x_{3})$ is the equation for the shock front. Then, due to the symmetry of $\mathcal{C}_{\sigma_1}^{\sigma_2}$, it is sufficient to study our problem in the domain $\mathcal{D}\setminus \mathcal{C}_{\sigma_1}^{\sigma_2}$.

\begin{figure}[H]
	\centering
	\begin{tikzpicture}[smooth, scale=1]
	\draw  [-latex] (8.25,1.45)--(3.75,4.15) node [left] {\footnotesize$x_{2}$};
	\draw  [densely dotted] (6.5,6.1)--(7,5.8);
        \draw  [densely dotted] (10,4)--(10.5,3.7);
	\draw  (6.5,6.1)--(6,2.8)--(10.5,3.7);
	\draw  [densely dotted] (8.5,5.8)--(8.5,6.2);
	\draw  [densely dotted] (8.5,3.6)--(8.5,4);
	\draw  (8.5,6.2)--(6,2.8);
	\draw (6.5,6.1) to [out=40,in=155]  (8.5,6.2) to [out=-25,in=100]  (10.5,3.7);
	\draw [dashed] (6.5,6.1) to [out=-80,in=165]  (8.5,3.6) to [out=-15,in=-180] (9.2,3.44);
        \draw(9.2,3.44) to [out=0,in=-150] (10.5,3.7);
        \draw [densely dotted](7,5.8)--(8.5,5.8)--(10,4);
        \draw [densely dotted] (7,5.8)--(8.5,4)--(10,4);
        \draw  [densely dotted]  (7,5.8)--(6,2.8)--(8.5,5.8);  
        \draw  [densely dotted]  (8.5,4)--(6,2.8)--(10,4);
        \draw  [densely dotted] (8.5,5.8)--(8.5,6.2);
	\draw  [densely dotted] (8.5,3.6)--(8.5,4);
        \draw (6.15,3.2)arc (90:138:0.65);
	\node at (5.55,3.45){\tiny{$\frac{\pi}{2}-\sigma_1$}};
	\draw (7.15,3.35)arc (15:40:0.72);
	\node at (7.28,3.6){\tiny$\sigma_2$};
	\draw  [-latex] (6,2.8)--(6,2) node [below] {\footnotesize$x_{1}$};
	\node at (6.2,2.4){\footnotesize$O$};
	\draw (6,2.8)--(5.5,2.4);
        \draw  [dashed]  (5.8,2.62)--(8.73,5.1);
	\draw  [densely dotted]  (8.73,5.1)--(9.4,5.656);
	\draw  [-latex]  (9.4,5.656)--(10.07,6.212)node [right] {\footnotesize$x_{3}$};
	\draw [thick, -stealth] (5,2) node [left] {\footnotesize 
        $\bm{v}_{\infty}$}  --(5.8,2.62) ;
	\draw [-latex](9.73,4.3)--(10.8,4.6) node[right] {\footnotesize{conical wing}};
	\draw [-latex](8.3,6.3)--(8.8,6.8) node[above right] {\footnotesize{shock}};
	\end{tikzpicture}
	\caption{Supersonic flow past a conical wing with diamond-shaped cross sections.}
	\label{fg1}
\end{figure}

From the property $(1)$ of the Chaplygin gas as above, we can introduce a potential function \mbox{$\Phi$} defined by \mbox{$\nabla_{\bm{x}}\Phi:=\bm{v}=(v_1,v_2,v_3)$, where $\bm{x}:=(x_1,x_2,x_3)$}. The flow is governed by the conservation of mass and the Bernoulli law
\begin{align}
	\mathrm{div}_{\bm{x}}(\rho \nabla_{\bm{x}}\Phi)&=0,\label{eq: conservation of mass}\\
	\frac{1}{2}|\nabla_{\bm{x}}\Phi|^{2}+h(\rho)&=\frac12B_{\infty},\label{eq: Bernoulli law}
\end{align}
where~$\rho>0$~is unknown and represent the density; $B_{\infty}/2=(v_{3\infty}^2-c_{\infty}^2)/2$ is the Bernoulli constant. Here, from \eqref{eq:Chaplygin gas}, the specific enthalpy $h(\rho)$ satisfies
\begin{equation}\label{eq:for h}
	h(\rho)=-\frac{c^2}{2}=-\frac{A}{2\rho^2}.
\end{equation}
In what follows, we choose $B_{\infty}=1$ without loss of generality. Then, by \eqref{eq: Bernoulli law}--\eqref{eq:for h}, the density $\rho$ can be expressed as a function of $\Phi$, that is,
\begin{equation}\label{eq:2.29}
	\rho=\frac{\sqrt{A}}{\sqrt{|\nabla_{\bm{x}}\Phi|^{2}-1}}.
\end{equation}

Denote by $S$ the shock front. The function $\Phi$ satisfies the following Rankine--Hugoniot conditions
	\begin{align}
		[\rho \nabla_{\bm{x}}\Phi]\cdot\bm{n}_{s}=0\quad &\text{on $S$},\label{eq: R-H condition}\\	
            [\Phi]=0\quad &\text{on $S$},\label{eq: contious of phi}
	\end{align}
where $[\cdot]$ stands for the jump of quantities across $S$, and $\bm{n}_{s}$ is the exterior normal to $S$. Note that the condition \mbox{\eqref{eq: R-H condition}} is naturally satisfied and thus compatible with the above property \mbox{$(2)$} of the Chaplygin gas.

Therefore, our problem can be formulated mathematically as the following boundary value problem.

\begin{problem}\label{prob1}
	For the conical wing $\mathcal{C}_{\sigma_1}^{\sigma_2}$ and oncoming flow given above, in the case that the shock front $S$ is only attached to the apex but away from the leading edges of $\mathcal{C}_{\sigma_1}^{\sigma_2}$, we aim to seek a solution $\Phi$ of system \eqref{eq: conservation of mass}--\eqref{eq: Bernoulli law} in the domain $\mathcal{D}\setminus \mathcal{C}_{\sigma_1}^{\sigma_2}$ with the Dirichlet boundary condition \eqref{eq: contious of phi} and the following slip boundary conditions
	\begin{align}
		\nabla_{\bm{x}}\Phi\cdot\bm{n}_{w}=0 \quad&\text{on $\{x_{1}=(x_2\cot\sigma_1-x_3)\tan\sigma_2,x_{2}>0,x_{3}>0\}$},\label{eq3:1.2}\\	
        \nabla_{\bm{x}}\Phi\cdot\bm{n}_{sy1}=0 \quad&\text{on $\{x_1=0\}\cap\overline{\mathcal{D}\setminus \mathcal{C}_{\sigma_1}^{\sigma_2}}$},\label{eq2:1.2}\\
		\nabla_{\bm{x}}\Phi\cdot\bm{n}_{sy2}=0\quad &\text{on $\{x_2=0\}\cap\overline{\mathcal{D}\setminus \mathcal{C}_{\sigma_1}^{\sigma_2}}$},\label{eq1:1.2}
	\end{align}
    where $\bm{n}_{w}=(1,-\cot\sigma_1\tan\sigma_2,\tan\sigma_2)$ is the exterior normal to $\{x_{1}=(x_2\cot\sigma_1-x_3) \tan\sigma_2,\\x_{2}>0, x_3>0\}$, and $\bm{n}_{sy1}=(1,0,0)$ is the exterior normal to $\{x_1=0\}$, and $\bm{n}_{sy2}=(0,-1,0)$ is the exterior normal to $\{x_2=0\}$.
\end{problem}

Our main result of this paper is the following theorem.

\begin{theorem}[Main Theorem]\label{thm: Main}
	Suppose that the conical wing $\mathcal{C}_{\sigma_1}^{\sigma_2}$ is defined by \eqref{eq:conical body}, and the state $(\rho_{\infty},\bm{v}_{\infty}=(0,0,v_{3\infty}))$ of the oncoming flow is uniform and supersonic. Then there exists a critical angle $\sigma_{\infty}=\sigma_{\infty}(\rho_{\infty},\bm{v}_{\infty})\in(0,{\pi}/{2})$ such that for any $\sigma_{1},\sigma_{2}\in(0,\sigma_{\infty})$, there exists a piecewise smooth solution to Problem $\ref{prob1}$.
\end{theorem}

Due to the geometric characteristics of $\mathcal{C}_{\sigma_1}^{\sigma_2}$, we can reduce Problem \ref{prob1} in conical coordinates $(\xi_1,\xi_2):=(x_1/x_3,x_2/x_3)$, as detailed in subsection \ref{sec:2.2}. This reformulation ultimately leads to an oblique derivative problem for a nonlinear degenerate elliptic equation (Problem \ref{prob2}). Since the conical wing $\mathcal{C}_{\sigma_1}^{\sigma_2}$ is not a small perturbation of a circular cone, the uniform ellipticity of this equation cannot be determined in advance. To overcome this difficulty, inspired by \cite{LY22,Serre09}, we will establish a crucial Lipschitz estimate. Note that the works in \cite{LY22,Serre09} mainly focus on the Dirichlet problem. However, the corresponding boundary condition of \eqref{eq3:1.2} is oblique in the $(\xi_1,\xi_2)$-coordinates. Thus, the techniques developed in these works for deriving the Lipschitz estimate are not directly applicable to our problem. 

We remark here that a closely related problem of shock reflection has been studied in \cite{CFe10,CFe18}, which is a two-dimensional unsteady case and can also be reformulated as a boundary value problem for a nonlinear degenerate elliptic equation.

The rest of the paper is structured into two sections. In Section \ref{sec:2}, we analyze the global structure of the shock wave in the $(\xi_1,\xi_2)$-coordinates, and then reduce Problem \ref{prob1} to an oblique derivative problem for a nonlinear degenerate elliptic equation, i.e., Problem \ref{prob2}. In Section \ref{sec:3}, we show the existence for Problem \ref{prob2} by the continuity method and vanishing viscosity method. To be specific, by introducing a new potential function, we first establish a Lipschitz estimate for a viscosity solution, which can guarantee the local uniform ellipticity of equation \eqref{eq: 2-D eq.} below. Then, we show the existence of the viscosity solution by the method of continuity, and further prove that the limit of the viscosity solution is the solution to Problem \ref{prob2}. In Appendix \ref{sec:appendix a}, we prove that for the Chaplygin gas, equation \eqref{eq: 2-D eq.} is elliptic in the interior of a parabolic-elliptic region.	

    \section{Preliminaries}\label{sec:2}
\subsection{The location of the shock}\label{sec:2.1}
Substituting \eqref{eq:2.29} into \eqref{eq: conservation of mass} yields a quasilinear equation for $\Phi$
\begin{multline}\label{eq:3-D for Phi}	
(c^{2}-\Phi^{2}_{x_{1}})\Phi_{x_{1}x_{1}}+(c^{2}-\Phi^{2}_{x_{2}})\Phi_{x_{2}x_{2}}+(c^{2}-\Phi^{2}_{x_{3}})\Phi_{x_{3}x_{3}}\\-2\Phi_{x_{1}}\Phi_{x_{2}}\Phi_{x_{1}x_{2}}-2\Phi_{x_{1}}\Phi_{x_{3}}\Phi_{x_{1}x_{3}}-2\Phi_{x_{2}}\Phi_{x_{3}}\Phi_{x_{2}x_{3}}=0,
\end{multline}
where
\begin{equation}\label{eq:for c of Phi}
	c^2=|\nabla_{\bm{x}}\Phi|^{2}-1.
\end{equation}
The characteristic equation of $\eqref{eq:3-D for Phi}$ takes the form (see \cite[Lemma 1.1]{CY15})
\begin{equation}\label{eq:2.31}
	Q(\bm{\kappa})=c^{2}-|\nabla_{\bm{x}}\Phi\cdot\bm{\kappa}|^{2} \quad\text{for any}~\bm{\kappa}\in\mathbb{R}^{3}~\text{with}~|\bm{\kappa}|=1,
\end{equation}
then equation \eqref{eq:3-D for Phi} is hyperbolic in a supersonic region. Under the assumption that the oncoming flow is supersonic, the downstream flow remains supersonic owing to the above property $(2)$ of the Chaplygin gas, and then equation \mbox{\eqref{eq:3-D for Phi}} is hyperbolic in the domain $\mathcal{D}\setminus \mathcal{C}_{\sigma_1}^{\sigma_2}$. This fact means that there exists a Mach cone $\mathcal{C}_{\infty}$, determined by the oncoming flow, emanating from the apex of $ \mathcal{C}_{\sigma_1}^{\sigma_2}$. It follows from the above property $(2)$ of the Chaplygin gas that this Mach cone $\mathcal{C}_{\infty}$ coincides with the shock front $S$.

We then calculate the equation for $\mathcal{C}_{\infty}$. From \cite[Appendix B]{LY22}, the expression of Mach cone for equation \eqref{eq:3-D for Phi} is 
\begin{equation}\label{eq:Mach cone}
	|\nabla_{\bm{x}}\Phi|^{2}-\vert \nabla_{\bm{x}}\Phi\cdot\frac{\bm{x}}{|\bm{x}|}\vert^{2}=c^{2}\quad\text{for}~\bm{x}\in\mathbb{R}^{3}\setminus\{0\}.
\end{equation}
Inserting \eqref{eq:for c of Phi} and $\nabla_{\bm{x}}\Phi=\bm{v}_{\infty}=(0,0,v_{3\infty})$ into \eqref{eq:Mach cone}, we obtain
\begin{equation}\label{eq:Mach cone C}
	\mathcal{C}_{\infty}:\quad v_{3\infty}^2x^2_3=|\bm{x}|^2.
\end{equation}
To avoid the shock front attaching to the edges of $\mathcal{C}_{\sigma_1}^{\sigma_2}$, the angles $\sigma_1$ and $\sigma_2$ must be less than half of the vertex angle of $\mathcal{C}_{\infty}$, namely,
\begin{equation}\label{eq: for sigma limit}
	\sigma_\infty:=\arcsin\Big(\dfrac{\sqrt{v^2_{3\infty}-1}}{v_{3\infty}}\Big)>\max\{\sigma_1,\sigma_2\}.
\end{equation}

Note that the problem \eqref{eq: conservation of mass}--\eqref{eq: Bernoulli law} and \eqref{eq: contious of phi}--\eqref{eq1:1.2} is invariant under the following scaling
\begin{equation*}
	\bm{x}\longmapsto \varsigma\bm{x}, \quad 	(\rho,\Phi)\longmapsto\Big(\rho,\dfrac{\Phi}{\varsigma}\Big)\quad\quad \text{for} \quad\varsigma\neq 0.
\end{equation*}
Therefore, we can seek a solution with the form:
\begin{equation}\label{eq: conical tran.} 
	\rho(\bm{x})=\rho(\bm{\xi}), \quad \Phi(\bm{x})=x_3\phi(\bm{\xi}),
\end{equation}
where $\bm{\xi}:=(\xi_{1},\xi_{2})=({x_1}/{x_3},{x_2}/{x_3})$ is the conical coordinates. For notational simplicity, in the $(\xi_{1},\xi_{2})$-coordinates, we still use $\mathcal{C}_{\infty}$ for the corresponding curves of the Mach cone. From \eqref{eq:Mach cone C} and \eqref{eq: conical tran.}, it follows that
\begin{equation}\label{eq1:cone infity}
    \mathcal{C}_{\infty}:\quad v_{3\infty}^2=1+|\bm{\xi}|^2.
\end{equation}
Let $P_1$ and $P_2$ be the intersection points of $\mathcal{C}_{\infty}$ with the $\xi_2$-axis and the $\xi_1$-axis, respectively. Also, let $P_3$ and $P_4$ denote the intersection points of the conical wing $\mathcal{C}_{\sigma_1}^{\sigma_2}$ with the $\xi_1$-axis and the $\xi_2$-axis, respectively.

Therefore, we can draw the pattern of the shock as in Figure \ref{fg P}.

\begin{figure}[H]
	\centering
	\begin{tikzpicture}[smooth, scale=1]
    \draw  [-latex](0,2)--(6,2) node [right] {\footnotesize$\xi_{1}$};
	\draw  [-latex](4,0.7)--(4,6.5) node [above] {\footnotesize$\xi_{2}$};
	\draw  (1,2)   to [out=90, in =180] (4,5) node [above left] {\footnotesize$P_{1}$} ;
	\node  at (1,2) [below] {\footnotesize$P_{2}$};
    \node  at (2.5,2) [below] {\footnotesize$P_{3}$};
    \node  at (1.75,2) [below] {\footnotesize$\Gamma_{sy2}$};
    \node  at (3.95,4.25) [right] {\footnotesize$P_{4}$};
    \node  at (4,4.75) [right] {\footnotesize$\Gamma_{sy1}$};
	\node  at (2.3,3.3) {$\Omega$};
	\node  at (3,3.6) {\footnotesize$\Gamma_{py}$};
	\node  at (1.9,4.8) {\footnotesize$\Gamma_{cone}^{\infty}$};
	\draw [draw=gray, fill=gray, fill opacity=0.6](2.5,2)--(4,4.5)--(4,2)--(2.5,2);
	\node at(4,2) [below right] {\footnotesize$O$};
	\fill(4,5)circle(1.3pt);
	\fill(1,2)circle(1.3pt);
	\fill(4,4.5)circle(1.3pt);
	\fill(2.5,2)circle(1.3pt);
	\end{tikzpicture}
	\caption{Domain in the conical coordinates.}
	\label{fg P}
\end{figure}

\subsection{BVP for a nonlinear degenerate elliptic equation}\label{sec:2.2}
In this subsection, we will restate Problem \ref{prob1} in the $(\xi_1,\xi_2)$-coordinates. Before proceeding further, let us introduce some notations. Denote by $\Gamma_{cone}^{\infty}$ the arc $P_{1}P_{2}$; by $\Gamma_{sy1}$, $\Gamma_{sy2}$ and $\Gamma_{py}$, respectively, the lines $P_{1}P_{4}$, $P_{2}P_{3}$ and $P_{3}P_{4}$. Let $\Omega$ be the domain $P_{1}P_{2}P_{3}P_{4}$ (see Figure \ref{fg P}). Also, let
\begin{equation*}
	D^{2}f[\bm{a},\bm{b}]:= \sum^{2}_{i,j=1} a_{i}b_{j}\partial_{ij}f\quad\text{for any $f\in C^2$ and $\bm{a},\bm{b}\in\mathbb{R}^{2}$}.
\end{equation*}

By \eqref{eq:3-D for Phi} and \eqref{eq: conical tran.}, the potential function $\phi$ satisfies
\begin{equation}\label{eq: 2-D eq.}
	c^2(\Delta\phi+D^{2}\phi[\bm{\xi},\bm{\xi}])-D^{2}\phi[D\phi-\chi\bm{\xi},D\phi-\chi\bm{\xi}]=0,
\end{equation}
where $D$ and $\Delta$ stand for the gradient and Laplacian operator with respect to $\bm{\xi}$, respectively, and
\begin{equation}\label{eq: velocity3}
	\chi=\phi-D\phi\cdot\bm{\xi}.
\end{equation}
Define a pseudo-Mach number
\begin{equation}\label{eq: 2D Mach num.}
	L^{2}:=\frac{|D\phi|^{2}+|\chi|^{2}-\frac{\phi^{2}}{1+|\bm{\xi}|^{2}}}{c^2}=1+\frac{1-\frac{\phi^{2}}{1+|\bm{\xi}|^{2}}}{c^2},
\end{equation}
where
\begin{equation}\label{eq:2.34}
	c^2=|D\phi|^{2}+|\chi|^{2}-1.
\end{equation}
Then, the type of equation \eqref{eq: 2-D eq.} is determined by $L$. More specifically, equation \eqref{eq: 2-D eq.} is elliptic when $0\leq L<1$, hyperbolic when $L>1$ and parabolic degenerate when $L=1$. 

From \eqref{eq1:cone infity} and \eqref{eq: 2D Mach num.}, together with $\phi_\infty=v_{3\infty}$, we see that equation \eqref{eq: 2-D eq.} is hyperbolic outside the domain $\Omega$ and parabolic degenerate on the boundary $\Gamma_{cone}^{\infty}$. Additionally, the analysis in Appendix \ref{sec:appendix a} ensures that equation \eqref{eq: 2-D eq.} is elliptic in the interior of a parabolic-elliptic region. It is therefore natural to require that equation \eqref{eq: 2-D eq.} is elliptic in the domain $\overline{\Omega}\setminus\overline{\Gamma_{cone}^{\infty}}$. However, the uniform ellipticity of equation \eqref{eq: 2-D eq.} in this domain cannot be obtained in advance.

Since equation \eqref{eq: 2-D eq.} is parabolic degenerate on $\Gamma_{cone}^{\infty}$ (i.e., $L=1$), the boundary condition on $\Gamma_{cone}^{\infty}$ can be simplified as
\begin{equation}\label{eq:3.1}
	\phi=\sqrt{1+|\bm{\xi}|^{2}}.
\end{equation}
Now, Problem \ref{prob1} can be reformulated as

\begin{problem}\label{prob2}
    For the conical wing $\mathcal{C}_{\sigma_1}^{\sigma_2}$ and oncoming flow given above, when $\sigma_{1},\sigma_{2}\in(0,\sigma_{\infty})$, we expect to find a solution $\phi$ satisfying the following oblique derivative problem:
	\begin{equation}\label{eq: problem for phi}
		\begin{cases}
			\text{Equation}~\eqref{eq: 2-D eq.}\quad&\text{in $\Omega$},\\
			\phi=\sqrt{1+|\bm{\xi}|^{2}}\quad&\text{on $\Gamma_{cone}^{\infty}$},\\
			D\phi\cdot\bm{\nu}_{py}+\chi\tan\sigma_2=0\quad&\text{on $\Gamma_{py}$},\\
			D\phi\cdot\bm{\nu}_{sy1}=0\quad&\text{on $\Gamma_{sy1}$},\\
                D\phi\cdot\bm{\nu}_{sy2}=0\quad&\text{on $\Gamma_{sy2}$},
		\end{cases}
	\end{equation}
    where $\bm{\nu}_{py}=(1,-\cot\sigma_1\tan\sigma_2)$, $\bm{\nu}_{sy1}=(1,0)$ and $\bm{\nu}_{sy2}=(0,-1)$ are the exterior normals to $\Gamma_{py}$, $\Gamma_{sy1}$ and $\Gamma_{sy2}$, respectively; the angle $\sigma_{\infty}$ is given by \eqref{eq: for sigma limit}; $\chi$ is defined by \eqref{eq: velocity3}.
\end{problem}

To prove Theorem \ref{thm: Main}, it suffices to show that

\begin{theorem}\label{thm4}
	Let $\sigma_{\infty}$ be as in Problem $\ref{prob2}$. Then, for any $\sigma_{1},\sigma_{2}\in(0,\sigma_{\infty})$, there exists a solution $\phi$ to Problem $\ref{prob2}$ satisfying
	\begin{equation*}
		\phi\in Lip(\overline{\Omega})\cap C^{1,\alpha}(\overline{\Omega}\setminus\overline{\Gamma_{cone}^{\infty}})\cap C^\infty\big(\overline{\Omega}\setminus(\overline{\Gamma_{cone}^{\infty}}\cup\{P_3\}\cup\{P_4\})\big)
	\end{equation*}
	and
	\begin{equation*}
		\phi >\sqrt{1+|\bm{\xi}|^{2}} \quad\text{in}~ \overline{\Omega}\setminus\overline{\Gamma_{cone}^{\infty}},
	\end{equation*}
    where $\alpha=\alpha(\sigma_1,\sigma_2)\in (0,1)$ is a constant.
\end{theorem}

\begin{remark}\label{non-existence}
When the oncoming flow passes the conical wing $\mathcal{C}_{\sigma_1}^{\sigma_2}$ with an attack angle, the shock front $S$ is no longer a circular cone, as rigorously analyzed in subsection $\ref{sec:2.1}$. This leads to the study of an elliptic equation in a concave domain containing a corner larger than $\pi$. For such a problem, the Lipschitz regularity of solutions near the corners cannot be expected (see, for instance, \cite{X20,L26}). Thus, the existence of solutions is difficult to obtain. We leave this for future work.
\end{remark}

\section{The Proof of Theorem \ref{thm4}}\label{sec:3}
\subsection{The strategy of the proof}\label{sec:3.2.2}
This subsection presents the proof strategy for Theorem \ref{thm4}. Consider the following problem:
\begin{equation}\label{eq: for Fmu}
	\mathcal{G}(\mu,\phi):=c^{2}(\Delta\phi+D^{2}\phi[\bm{\xi},\bm{\xi}])-\mu D^{2}\phi[D\phi-\chi\bm{\xi},D\phi-\chi\bm{\xi}]=0\quad\text{in $\Omega$}
\end{equation}
and 
	\begin{equation}\label{eq:nianxing BC}
		\begin{cases}
	\phi=\sqrt{1+|\bm{\xi}|^{2}}+\varepsilon \quad&\text{on $\Gamma_{cone}^{\infty}$},\\ 
			D\phi\cdot\bm{\nu}_{py}+\chi\tan\sigma_2=0\quad&\text{on $\Gamma_{py}$},\\
			D\phi\cdot\bm{\nu}_{sy1}=0\quad&\text{on $\Gamma_{sy1}$},\\
                D\phi\cdot\bm{\nu}_{sy2}=0\quad&\text{on $\Gamma_{sy2}$},	
		\end{cases}
	\end{equation}
where $\mu\in[0,1]$ and $\varepsilon>0$ are parameters. Here, motivated by the work in \cite{Serre09}, we employ the vanishing viscosity method to deal with the degenerate boundary. Also note that for $\mu=0$, equation \eqref{eq: for Fmu} reduces to a linear elliptic equation, while for $\mu=1$, it becomes the original equation \eqref{eq: 2-D eq.}. This structural property inspires us to solve Problem \ref{prob2} through the continuity method. 

We turn to analyze the ellipticity of equation \eqref{eq: for Fmu}. Let $\phi_{\mu,\varepsilon}$ be the solution to problem \eqref{eq: for Fmu}--\eqref{eq:nianxing BC}. It is straightforward to verify that for each $\mu\in(0,1]$, equation \eqref{eq: for Fmu} is elliptic if
\begin{equation*}
	\phi_{\mu,\varepsilon}>\sqrt{(1+|\bm{\xi}|^2)\Big(1+\frac{\mu-1}{\mu}c^2\Big)}.
\end{equation*}
Thus, the inequality $\phi_{\mu,\varepsilon} > \sqrt{1+|\bm{\xi}|^2}$ ensures the ellipticity of \eqref{eq: for Fmu} for all $\mu\in[0,1]$. In addition, for any $\mu\in[0,1]$, equation \eqref{eq: for Fmu} is uniformly elliptic if
\begin{align}
	\sqrt{1+|\bm{\xi}|^2}+\delta_{0}\leq  \phi_{\mu,\varepsilon}&<C,\label{eq: estimate phi}\\
	\vert D\phi_{\mu,\varepsilon}\vert&<C,\label{eq: lip estimate phi}
\end{align}
where $\delta_{0}$ is a positive number and $C$ is a bounded constant.

Define a set
\begin{equation}\label{eq: J varepsilon}
	\begin{aligned}
		J_{\varepsilon}:=\{\mu\in[0,1]:~&\text{the function }\phi_{\mu,\varepsilon} ~\text{satisfies}~\eqref{eq: for Fmu}\text{-}\eqref{eq:nianxing BC}~\text{with}~\phi_{\mu,\varepsilon}\geq \sqrt{1+\vert\bm{\xi}\vert^2}+\varepsilon~\text{in}\\		&\overline{\Omega}\setminus\overline{\Gamma_{cone}^{\infty}}~\text{and}~\phi_{\mu,\varepsilon}\in C^{1}(\overline{\Omega})\cap C^2\big(\overline{\Omega}\setminus(\{P_3\}\cup\{P_4\})\big)\}.
	\end{aligned}
\end{equation}
We now outline the strategy of our proof. For any fixed $\varepsilon>0$, we apply the continuity method to prove $J_{\varepsilon}=[0,1]$. To be specific, we demonstrate that $J_{\varepsilon}$ is open, closed and $0\in J_{\varepsilon}$. The crucial step here is to derive the uniform estimates \eqref{eq: estimate phi}--\eqref{eq: lip estimate phi}, which are independent of both the parameters $\mu$ and $\varepsilon$. Subsequently, we prove that the limiting function $\mathop{\mathrm{lim}}\limits_{\varepsilon \to 0^+} \phi_{1,\varepsilon}$ is the solution to Problem \ref{prob2}.

\subsection{Lipschitz estimate}\label{sec: prior estimate}
In this subsection, we omit the subscripts of $\phi_{\mu,\varepsilon}$ and $\Phi_{\mu,\varepsilon}$ for convenience. Due to the oblique derivative boundary condition on $\Gamma_{py}$, i.e., \eqref{eq:nianxing BC}$_2$, the techniques utilized in \cite{LY22} are not applicable to deriving the estimate \eqref{eq: estimate phi}. To overcome this difficulty, we reformulate the problem \eqref{eq: for Fmu}--\eqref{eq:nianxing BC} in the spherical coordinates, defined as
\begin{equation*}
    (r,\theta,\varphi):= \big(\sqrt{x^2_1+x^2_2+x^2_3}, \arccos{(x_1/r)}, \arctan{({x_3}/{x_2})}\big),
\end{equation*}
and introduce a function with the form
\begin{align}\label{eq: spherical tran.}
    \rho(\bm{x})=\rho(\bm{\zeta}), \quad \Phi(\bm{x})=r\psi(\bm{\zeta}),
\end{align}
where~$\bm{\zeta}:=(\theta, \varphi)$. In the spherical coordinates, the boundary condition on $\Gamma_{py}$ becomes a Neumann condition (see \eqref{eq:nianxing BC4}$_2$). This fact permits the construction of suitable sub-solutions and super-solutions to establish the $L^\infty$--estimate for $\psi$ (Lemma \ref{lemma: inf-estimate}), thereby yielding the estimate \eqref{eq: estimate phi}.

By abuse of notation but without misunderstanding, in the spherical coordinates, we still use $\Omega$ for the corresponding domain, $\Gamma_{cone}^{\infty},\Gamma_{py},\Gamma_{sy1}$ and $\Gamma_{sy2}$ for the corresponding curves of the boundaries, $\bm{\nu}_{py},\bm{\nu}_{sy1}$ and $\bm{\nu}_{sy2}$ for the corresponding exterior normals.

By the boundedness of $\Omega$, the parameters $\varepsilon$ and $\sqrt{1+|\bm{\xi}|^2} \varepsilon$ are not distinguished, and are collectively denoted by $\varepsilon$ throughout this paper. From \eqref{eq: conical tran.}, \eqref{eq: for Fmu}--\eqref{eq:nianxing BC} and \eqref{eq: spherical tran.}, a trivial calculation yields that the function~$\psi$~satisfies
\begin{multline}\label{eq: psi zhankai}
    c^{2}\big(\dfrac{\partial_{\theta}(\sin\theta\partial_{\theta}\psi)}{\sin\theta}+\dfrac{\partial_{\varphi\varphi} \psi}{\sin^2\theta}+2\psi\big)-\mu\Big(\partial^{2}_\theta\psi\partial_{\theta\theta}\psi+\dfrac{\partial^2_\varphi\psi}{\sin^2\theta}\dfrac{\partial_{\varphi\varphi} \psi}{\sin^2\theta}\\+2\partial_\theta \psi\dfrac{ \partial_\varphi\psi}{\sin\theta}\dfrac{\partial_{\theta\varphi}\psi}{\sin\theta}-\cot\theta\dfrac{\partial^2_\varphi\psi}{\sin^2\theta}\partial_\theta\psi+\big(\partial^{2}_\theta\psi+\dfrac{\partial^2_\varphi\psi}{\sin^2\theta}\big)\psi\Big)
    =0\quad\text{in $\Omega$}
\end{multline}
with 
  \begin{equation}\label{eq:nianxing BC4}
		\begin{cases}
	        \psi=1+\varepsilon \quad&\text{on $\Gamma_{cone}^{\infty}$},\\ 
			D_{\bm{\zeta}}\psi\cdot\bm{\nu}_{py}=0\quad&\text{on $\Gamma_{py}$},\\
			D_{\bm{\zeta}}\psi\cdot\bm{\nu}_{sy1}=0\quad&\text{on $\Gamma_{sy1}$},\\
                D_{\bm{\zeta}}\psi\cdot\bm{\nu}_{sy2}=0\quad&\text{on $\Gamma_{sy2}$},	
		\end{cases}
	\end{equation}
where $D_{\bm{\zeta}}:=(\partial_\theta, \dfrac{1}{\sin\theta}\partial_\varphi)$~represents the gradient operator with respect to $\bm{\zeta}$.     

Denote by $\mathcal{N}_\mu \psi=0$ the equation \eqref{eq: psi zhankai} for simplicity. To obtain the $L^\infty$--estimate of $\psi$, we begin by giving the following comparison principle.

\begin{lemma}\label{lemma: comparison principle}
Let $\Omega_{D}$ be an open bounded domain in the unit sphere that does not contain $\{\theta=0\}$, with piecewise smooth boundary $\partial\Omega_D$ composed of a finite number of curves $\Gamma_i$, where $i=1,\cdots,n$. Also assume that if $\Gamma_k$ and $\Gamma_l$ meet, the formed angle lies in the interval $(0,\pi]$. Suppose that ${\psi}_{\pm}\in C^{0}(\overline{\Omega_{D}})\cap C^{1}(\overline{\Omega_{D}}\setminus\overline{\Gamma_1})\cap C^2(\Omega_{D})$ satisfy ${\psi}_{\pm}>1$ in $\overline{\Omega_{D}}\setminus\overline{\Gamma_1}$. Assume further that for any $\mu\in [0,1]$, the operator $\mathcal{N}_\mu$ is locally uniformly elliptic with respect to either ${\psi}_{+}$ or $\psi_-$, and
	\begin{align*}
	\mathcal{N}_\mu\psi_-\geq0, \quad\mathcal{N}_\mu\psi_+\leq0\quad&\text{in}~\Omega_{D},\\
   \psi_-\leq\psi_+\quad&\text{on}~\Gamma_1,\\
     D_{\bm{\zeta}}\psi_-\cdot\bm{\nu}_{d_j}<D_{\bm{\zeta}}\psi_+\cdot\bm{\nu}_{d_j}\quad&\text{on}~\Gamma_j,
	\end{align*}
where $\bm{\nu}_{d_j}$~is the exterior normal to $\Gamma_j$ with $j=2,\cdots,n$. Then, it follows that ${\psi}_{-}\leq{\psi}_{+}$ in $\overline{\Omega_{D}}\setminus\overline{\Gamma_1}$.
\end{lemma}

\begin{proof}
    Let
	\begin{equation*}
		\bar{\psi}:=\psi_- -\psi_+.
	\end{equation*}
    Since ${\psi}_{\pm}> 1$ in $\overline{\Omega_{D}}\setminus\overline{\Gamma_1}$, we are able to define a function $z>0$ by
	\begin{equation}\label{eq:3.6}
		{\psi}(\bm{\zeta})=\cosh z(\bm{\zeta}).
	\end{equation}
    A direct computation gives
	\begin{multline}\label{eq:3.9}	
    \big(1+|D_{\bm{\zeta}}z|^2-\mu\partial^{2}_\theta z\big)\partial_{\theta\theta}z+(1+|D_{\bm{\zeta}}z|^2-\mu\dfrac{\partial^2_\varphi z}{\sin^2\theta})\dfrac{\partial_{\varphi\varphi} z}{\sin^2\theta}-2\mu\partial_\theta z\dfrac{ \partial_\varphi z}{\sin\theta}\dfrac{\partial_{\theta\varphi}z}{\sin\theta}\\+
    \cot\theta\big(1+|D_{\bm{\zeta}}z|^2+\mu\dfrac{\partial^2_\varphi z}{\sin^2\theta}\big)\partial_\theta z+(2+(1-\mu)|D_{\bm{\zeta}}z|^2)\frac{1+|D_{\bm{\zeta}}z|^2}{\tanh z}=0.
	\end{multline}
    The principal coefficients of \eqref{eq:3.9} is independent of $z$. Besides, the lower-order term of this equation is non-increasing in $z$ for any fixed $(\bm{\zeta},Dz)\in \Omega_{D}\times \mathbb{R}^{2}$. Thus, with the structural feature of \eqref{eq:3.9}, and using the same argument as \cite[Lemma \ref{lemma: comparison principle}]{LY22}, we have 
	\begin{equation}\label{eq: for assert}
\sup_{\Omega_D}\bar{\psi}\leq\sup_{\partial\Omega_D}\bar{\psi}.
	\end{equation}
    Moreover, the condition $D_{\bm{\zeta}}\psi_-\cdot\bm{\nu}_{d_j}<D_{\bm{\zeta}}\psi_+\cdot\bm{\nu}_{d_j}$ on $\Gamma_j$ and the assumption of the regularity of  ${\psi}_{\pm}$ imply that the function $\bar{\psi}$ cannot attain its maximum on $\Gamma_j$ and the corners, i.e.,
    \begin{equation}\label{eq:bd est}		  \sup_{\Omega_D}\bar{\psi}\leq\sup_{\partial\Omega_D}\bar{\psi}\leq\sup_{\Gamma_1}\bar{\psi}.
	\end{equation}
  Therefore, from $\psi_-\leq\psi_+$ on $\Gamma_1$, one knows $\bar{\psi}\leq 0$ in $\overline{\Omega_{D}}\setminus\overline{\Gamma_1}$. This completes the proof.
\end{proof}

We next utilize the above lemma to derive the $L^\infty$--estimate for $\psi$.
\begin{lemma}\label{lemma: inf-estimate}
	Let $\psi \in C^{0}(\overline{\Omega})\cap C^{1}(\overline{\Omega}\setminus\overline{\Gamma_{cone}^{\infty}})\cap C^2\big(\overline{\Omega}\setminus(\overline{\Gamma_{cone}^{\infty}}\cup\{P_3\}\cup\{P_4\})\big)$ satisfy \eqref{eq: psi zhankai}--\eqref{eq:nianxing BC4} and $\psi\geq 1+\varepsilon$ in $\overline{\Omega}\setminus\overline{\Gamma_{cone}^{\infty}}$. Then there exists a positive constant $C$, which is independent of $\mu$ and $\varepsilon$, such that
	\begin{equation}\label{eq1:3.13}
		1+\varepsilon<{\psi}\leq C\quad\text{in $\overline{\Omega}\setminus\overline{\Gamma_{cone}^{\infty}}$}.
	\end{equation}
    Moreover, for any subdomain $\Omega_{sub}$ that is strictly away from the degenerate boundary $\Gamma_{cone}^{\infty}$, we have
        \begin{equation}\label{eq:strictly bounded}
		1+\varepsilon+\delta_0<{\psi}\leq C,
	\end{equation}
    where the constant $\delta_0>0$ depends only on the domain $\Omega$. 
\end{lemma}

\begin{proof}
	 First, we find an exact solution to equation \eqref{eq: psi zhankai}. From \eqref{eq: conical tran.} and \eqref{eq: for Fmu}, the corresponding $\Phi$ satisfies
	\begin{equation}\label{eq:3.25}
		(|\nabla_{\bm{x}}\Phi|^2-1)\Delta_{\bm{x}}\Phi-\mu D^2_{\bm{x}}\Phi[\nabla_{\bm{x}}\Phi,\nabla_{\bm{x}}\Phi]=0.
	\end{equation}
     Clearly, for any constant vector $\bm{\eta}\in \mathbb{R}^{3}$, the linear function $\Phi^{\bm{\eta}}=\bm{\eta}\cdot\bm{x}$ is a solution to \eqref{eq:3.25}. We deduce from \eqref{eq: spherical tran.} that
	\begin{equation}\label{eq:3.26}
		{\psi}^{\bm{\eta}}(\bm{\zeta})=\bm{\eta}\cdot\frac{\bm{x}}{|\bm{x}|}
	\end{equation}
	is an exact solution to equation \eqref{eq: psi zhankai}. In addition, since
	\begin{align}\label{eq: d phi first}
            \begin{split}
	    \partial_\theta {\psi}^{\bm{\eta}}(\bm{\zeta})&=-\eta_1\sin\theta+\eta_2\cos\theta\cos\varphi+\eta_3\cos\theta\sin\varphi,\\
            \dfrac{\partial_\varphi {\psi}^{\bm{\eta}}(\bm{\zeta})}{\sin\theta}&=-\eta_2 \sin\varphi+\eta_3\cos\varphi,
            \end{split}
	\end{align}	
    the function ${\psi}^{\bm{\eta}}$ is Lipschitz bounded in the domain $\Omega$.
    
    Next, we construct the super-solutions to problem \eqref{eq: psi zhankai}--\eqref{eq:nianxing BC4}. Let $P^{\bm{\eta}}:=\frac{\bm{\eta}}{|\bm{\eta}|}$ denote the intersection of the vector $\bm{\eta}$ with the unit sphere, and let
    \begin{equation*}
         \Lambda:=\{(\theta,\varphi):\theta\in(0,\frac{\pi}{2}),\varphi\in(\frac{3\pi}{2},2\pi)\}.
    \end{equation*}
    Consider the following set with a fixed $\varepsilon>0$
	\begin{equation}\label{eq:4.8}
		\Sigma_{+}:=\{\bm{\eta}\in\mathbb{R}^{3}: P^{\bm{\eta}}\in \Lambda, \text{ and }{\psi}^{\bm{\eta}}>1+\varepsilon\text{ on }\Gamma_{cone}^{\infty}\}.
	\end{equation}
    Note that for any fixed $\bm{\eta}\in \mathbb{R}^{3}$, the function ${\psi}^{\bm{\eta}}$ is monotonically decreasing in the angle between $\frac{\bm{x}}{|\bm{x}|}$ and $\bm{\eta}$. Hence, when $P^{\bm{\eta}}\in \Lambda$, the values of $D_{\bm{\zeta}}\psi\cdot\bm{\nu}_{py}$ on $\Gamma^{\infty}_{cone}$, $D_{\bm{\zeta}}\psi\cdot\bm{\nu}_{sy1}$ on $\Gamma_{sy1}$, and $D_{\bm{\zeta}}\psi\cdot\bm{\nu}_{sy2}$ on $\Gamma_{sy2}$ are strictly positive. This indicates that
   when $\bm{\eta}\in \Sigma_{+}$, the function ${\psi}^{\bm{\eta}}$ satisfies
	\begin{equation*}
		\begin{cases}
			\mathcal{N}_\mu {\psi}^{\bm{\eta}}= 0\quad&\text{in $\Omega$},\\
			{\psi}^{\bm{\eta}}>1+\varepsilon\quad&\text{on $\Gamma_{cone}^{\infty}$},\\
            D_{\bm{\zeta}}\psi\cdot\bm{\nu}_{py}>0\quad&\text{on $\Gamma_{py}$},\\
			D_{\bm{\zeta}}\psi\cdot\bm{\nu}_{sy1}>0\quad&\text{on $\Gamma_{sy1}$},\\
                D_{\bm{\zeta}}\psi\cdot\bm{\nu}_{sy2}>0\quad&\text{on $\Gamma_{sy2}$}.
		\end{cases}
	\end{equation*}
   Additionally, it follows from \eqref{eq:3.26} that ${\psi}^{\bm{\eta}}>1+\varepsilon$ holds in $\overline{\Omega}\setminus\overline{\Gamma_{cone}^{\infty}}$. Combining this fact and \eqref{eq: d phi first}, the operator $\mathcal{N}_\mu $ is locally uniformly elliptic with respect to ${\psi}^{\bm{\eta}}$. Applying Lemma \ref{lemma: comparison principle}, we conclude that $\psi\leq {\psi}^{\bm{\eta}}$ in $\overline{\Omega}\setminus\overline{\Gamma_{cone}^{\infty}}$. Define $\psi^+$ as the infimum of all such super-solutions. Then, $\psi\leq \psi^+$ holds in $\overline{\Omega}\setminus\overline{\Gamma_{cone}^{\infty}}$. Moreover, thanks to the convexity of $\Gamma_{cone}^{\infty}$, we obtain
	\begin{equation*}
		\psi^+=1+\varepsilon\quad\text{on $\Gamma_{cone}^{\infty}$}.
	\end{equation*}
  
   Finally, we construct the sub-solutions to problem \eqref{eq: psi zhankai}--\eqref{eq:nianxing BC4}. Consider the following set with a fixed $\varepsilon>0$
	\begin{equation}\label{eq: Sigma fu}
		\Sigma_{-}:=\{\bm{\eta}\in\mathbb{R}^{3}: P^{\bm{\eta}}\in \Omega,\text{ and }{\psi}^{\bm{\eta}}<1+\varepsilon~\text{on}~\Gamma_{cone}^{\infty}\},
	\end{equation}
  Similarly, when $\bm{\eta}\in {\Sigma}_{-}$, the function ${\psi}^{\bm{\eta}}$ satisfies
	\begin{equation*}
		\begin{cases}
			\mathcal{N}_\mu {\psi}^{\bm{\eta}}= 0\quad&\text{in $\Omega$},\\
			{\psi}^{\bm{\eta}}<1+\varepsilon\quad&\text{on $\Gamma_{cone}^{\infty}$},\\
            D_{\bm{\zeta}}\psi\cdot\bm{\nu}_{py}<0\quad&\text{on $\Gamma_{py}$},\\
			D_{\bm{\zeta}}\psi\cdot\bm{\nu}_{sy1}<0\quad&\text{on $\Gamma_{sy1}$},\\
                D_{\bm{\zeta}}\psi\cdot\bm{\nu}_{sy2}<0\quad&\text{on $\Gamma_{sy2}$}.
		\end{cases}
	\end{equation*}
  From Remark \ref{remark:lip on supsolution} below, and using the super-solution $\psi^+$ above and the assumption of $\psi\geq 1+\varepsilon$ in $\overline{\Omega}\setminus\overline{\Gamma_{cone}^{\infty}}$, we deduce the locally uniform ellipticity of $\mathcal{N}_\mu$ with respect to $\psi$. Thus, using Lemma \ref{lemma: comparison principle}, we have
	\begin{equation}\label{eq1:11}
		{\psi}\geq {\psi}^{\bm{\eta}}\quad\text{in $\overline{\Omega}\setminus\overline{\Gamma_{cone}^{\infty}}$},
	\end{equation}
  for all $\bm{\eta}\in {\Sigma}_{-}$.  
   
    Define $\psi^-$ as the supremum of all such sub-solutions. We assert here that for any subdomain $\Omega_{sub}$ strictly away from the degenerate boundary $\Gamma_{cone}^{\infty}$, there exists a positive constant $\delta_0$, depending only on $\Omega_{sub}$, such that
	\begin{equation}\label{eq1:3.21}
		{\psi}^{-}\geq 1+\varepsilon+\delta_0\quad\text{in $\Omega_{sub}$}.
	\end{equation}
   Given any point~$\bm{x}_{0}=(x_{10},x_{20},x_{30})$ satisfying $\frac{(x_{10},x_{20},x_{30})}{r}\in \Omega_{sub}$, there exists a sufficiently small constant $0<\delta_0\ll 1$ such that the constant vector $\bm{\eta}_{0}=(1+\varepsilon+\delta_0)\frac{(x_{10},x_{20},x_{30})}{r}\in\Sigma_{-}$. The assertion then follows directly from the definition of $\psi^-$.
 
   Thanks to the continuity of ${\psi}^{-}$, the inequality \eqref{eq1:3.21} means ${\psi}^{-}>1+\varepsilon$ in $\overline{\Omega}\setminus\overline{\Gamma_{cone}^{\infty}}$ and ${\psi}^{-}\geq 1+\varepsilon$ on $\Gamma_{cone}^{\infty}$. Besides, we observe that ${\psi}^{-}\leq 1+\varepsilon$ on $\Gamma_{cone}^{\infty}$. This leads to
	\begin{equation}\label{eq:3.31}
		{\psi}^{-}= 1+\varepsilon\quad\text{on $\Gamma_{cone}^{\infty}$}.
	\end{equation}
 
  In conclusion, we get
	\begin{align}
		1+\varepsilon<{\psi}^-\leq \psi\leq \psi^+\quad&\text{in}~  \overline{\Omega}\setminus\overline{\Gamma_{cone}^{\infty}}\label{eq1:4.19}
		\shortintertext{and}
		\psi=\psi^\pm=1+\varepsilon\quad&\text{on}~ \Gamma_{cone}^{\infty}.\label{eq1:4.20}
	\end{align}
	The proof is complete. 
\end{proof}

\begin{remark}\label{remark:relax cd}
From the above discussion, the condition $\psi\geq 1+\varepsilon$ in $\bar{\Omega}\setminus\overline{\Gamma_{cone}^{\infty}}$ stated in Lemma $\ref{lemma: inf-estimate}$ is intended to ensure the locally uniform ellipticity of $\mathcal{N}_\mu$ with respect to $\psi$; otherwise, this condition can be relaxed to $\psi> 1$ in $\bar{\Omega}\setminus\overline{\Gamma_{cone}^{\infty}}$.
\end{remark}

Let $\phi^{\pm}:=\sqrt{1+|\bm{\xi}|^2} \psi^{\pm}$. Then, from \eqref{eq: conical tran.} and \eqref{eq: spherical tran.}, together with estimate \eqref{eq1:4.19}--\eqref{eq1:4.20}, we have
\begin{equation}\label{eq1:3.17}
	\sqrt{1+|\bm{\xi}|^2}+\varepsilon<\phi^-\leq{\phi}\leq\phi^+\leq C\quad\text{in $\overline{\Omega}\setminus\overline{\Gamma_{cone}^{\infty}}$}
\end{equation}
and
\begin{equation}\label{eq1:3.18}
	{\phi}={\phi}^{\pm}=\sqrt{1+|\bm{\xi}|^2}+\varepsilon\quad\text{on $\Gamma_{cone}^{\infty}$}.
\end{equation}
where the constant $C$ is independent of $\mu$ and $\varepsilon$. 

Now, we turn to establish the Lipschitz estimate \eqref{eq: lip estimate phi}.

\begin{lemma}\label{lemma: lip-estimate}
	Let $\phi \in C^{0}(\overline{\Omega})\cap C^{1}(\overline{\Omega}\setminus\overline{\Gamma_{cone}^{\infty}})\cap C^2\big(\overline{\Omega}\setminus(\overline{\Gamma_{cone}^{\infty}}\cup\{P_3\}\cup\{P_4\})\big)$ satisfy \eqref{eq: for Fmu}--\eqref{eq:nianxing BC} and ${\phi}\geq\sqrt{1+|\bm{\xi}|^2}+\varepsilon$ in $\overline{\Omega}\setminus\overline{\Gamma_{cone}^{\infty}}$. Then there exists a constant $C>0$, which is independent of $\mu$ and $\varepsilon$, such that
	\begin{equation}\label{eq:3.42} 
		\|D{\phi}\|_{L^{\infty}(\Omega)}\leq C.
	\end{equation}
\end{lemma}

\begin{proof}
    On the boundary $\Gamma_{cone}^{\infty}$, we infer from \eqref{eq1:3.17}--\eqref{eq1:3.18} that
	\begin{equation}\label{eq:3.34}
		\|D{\phi}\|_{L^{\infty}(\overline{\Gamma_{cone}^{\infty}})}\leq \|D{\phi}^{+}\|_{L^{\infty}(\overline{\Gamma_{cone}^{\infty}})}.
	\end{equation}     
    Besides, by applying the second to fourth boundary conditions in \eqref{eq:nianxing BC}, we have
        \begin{align}\label{eq: D for P3P4}
            \begin{split}
	    D\phi=&(\frac{\phi\tan\sigma_2}{\tan\sigma_2\xi_{1P_3}-1},0)\quad\text{at  $P_3(\xi_{1P_3},0)$},\\
            D\phi=&(0,\frac{\phi}{\cot\sigma_1-\xi_{2P_4}})\quad\text{at $P_4(0,\xi_{2P_4})$}.
            \end{split}
	\end{align}	 
    
    Note that the maximum of $|D\phi|^{2}$ cannot be attained at any interior point of the domain unless the function is constant, as demonstrated in \cite[Lemma 3.5]{LY22}. Also, it is clear that any point on $\Gamma_{sy1}\cup\Gamma_{sy2}$ can always be regarded as an interior point of $\Omega$. To deal with the point on $\Gamma_{py}$, we introduce a rotation transformation
      \begin{equation}\label{eq:4.4}	(\hat{x}_{1},\hat{x}_{2},\hat{x}_{3})=(x_{1}\cos\sigma_2+x_{3}\sin\sigma_2,x_{2},-x_{1}\sin\sigma_2+x_{3}\cos\sigma_2).
      \end{equation}
    The corresponding scaling is defined by
    \begin{equation}\label{eq:4.5}	(\hat{x}_1,\hat{x}_2,\hat{x}_3)\longmapsto(\hat{\xi}_{1},\hat{\xi}_{2}):=\Big(\frac{\hat{x}_1}{\hat{x}_3},\frac{\hat{x}_2}{\hat{x}_3}\Big), \quad
	(\rho,\Phi)\longmapsto(\rho,\hat{\phi}):=\Big(\rho, \frac{\Phi}{\hat{x}_3}\Big).
    \end{equation}
      
    The rotational invariance of equation \eqref{eq:3-D for Phi} ensures that the function $\hat{\phi}$ also satisfies equation \eqref{eq: for Fmu}. Under the rotation transformation \eqref{eq:4.4}, the corresponding point of $P_3$ in Figure \ref{fg P} coincides with the origin in the $(\hat{\xi}_{1},\hat{\xi}_{2})$-coordinates. Since equation \eqref{eq: for Fmu} is rotation-invariant and reflection-symmetric with respect to the coordinate axes, the corresponding points on $\Gamma_{py}$ can be treated as interior points in the $(\hat{\xi}_{1},\hat{\xi}_{2})$-coordinates. Furthermore, it follows from \eqref{eq: conical tran.} and \eqref{eq:4.5} that the estimates \eqref{eq1:3.17}--\eqref{eq1:3.18} and \eqref{eq:3.34} remain valid for $\hat{\phi}$. Therefore, we have
   	\begin{equation}\label{eq1:3.11}
		\|D\hat{\phi}\|_{L^{\infty}(\hat{\Omega})}\leq \|D\hat{\phi}\|_{L^{\infty}(\overline{\hat{\Gamma}_{cone}^{\infty}})},
	\end{equation} 
    where $\hat{\Omega}$ and $\hat{\Gamma}_{cone}^{\infty}$ denote the corresponding domain $\Omega$ and boundary $\Gamma_{cone}^{\infty}$, respectively, in the $(\hat{\xi}_{1},\hat{\xi}_{2})$-coordinates. This completes the proof.
\end{proof}

\begin{remark}\label{remark:lip on supsolution}
 Note that the establishment of the Lipschitz estimate \eqref{eq:3.42} depends only on the super-solution $\psi^+$ and is independent of the sub-solution $\psi^-$ obtained in Lemma $\ref{lemma: inf-estimate}$.
\end{remark}

\begin{remark}\label{remark:local uniform ellipticity}
Since the estimates \eqref{eq:strictly bounded} and \eqref{eq:3.42} are independent of $\mu$ and $\varepsilon$, equation \eqref{eq: 2-D eq.} is uniformly elliptic in any subdomain strictly away from the degenerate boundary $\Gamma_{cone}^{\infty}$.
\end{remark}

\subsection{The Proof of Theorem \ref{thm4}}\label{sec:3.2.5}
We are going to prove Theorem \ref{thm4} via the strategy outlined in subsection \ref{sec:3.2.2}. For a fixed $\varepsilon>0$, we show $J_{\varepsilon}=[0,1]$ through the following three steps.

\textit{Step} 1. $0\in J_{\varepsilon}$. Let us consider
\begin{equation}\label{eq:3.43}
	\mathcal{G}(0,\phi):=\Delta\phi+D^{2}\phi[\bm{\xi},\bm{\xi}]=0\quad\text{in $\Omega$}
\end{equation}
and the boundary conditions \eqref{eq:nianxing BC}. Using \eqref{eq: velocity3}, the boundary condition on $\Gamma_{py}$ can be rewritten as
	\begin{equation}\label{eq:condition on py}
		D\phi\cdot(\bm{\nu}_{py}-\tan\sigma_2\bm{\xi})+\phi\tan\sigma_2=0.
	\end{equation}
Also note that $\bm{\nu}_{py}=(1,-\cot\sigma_1\tan\sigma_2)$, then a simple calculation leads to the equation for $\Gamma_{py}$
        \begin{equation}\label{eq:equation of py}
		\xi_1-\xi_2\cot\sigma_1\tan\sigma_2+\tan\sigma_2=\bm{\xi}\cdot\bm{\nu}_{py}+\tan\sigma_2=0,
	\end{equation}
which implies
        \begin{equation}\label{eq:contrast}
		(\bm{\nu}_{py}-\tan\sigma_2\bm{\xi})\cdot \bm{\nu}_{py}=|\bm{\nu}_{py}|^2+\tan^2\sigma_2>0.
	\end{equation}
Thus, \eqref{eq:3.43} with \eqref{eq:nianxing BC} is an oblique derivative problem for a linear elliptic equation. 

Notice that $\bm{\nu}_{py}$ denotes the exterior normal to $\Gamma_{py}$ and $\tan\sigma_2>0$ due to $\sigma_2\in(0,\sigma_\infty)\subset(0,\frac{\pi}{2})$. Moreover, the angles at $P_3,P_4$ belong to $(\frac{\pi}{2},\pi)$ (see Figure \ref{fg P}). Then, the result in \cite[Theorem 1]{Lieb88} can be applied here; namely, there is a unique solution $\phi_{0,\varepsilon}\in C^{1}(\overline{\Omega})\cap C^{2}(\Omega)$. Note that  $\Gamma_{cone}^{\infty}$ is of $C^\infty$, and the points on $\Gamma_{py}\cup\Gamma_{sy1}\cup\Gamma_{sy2}$ can be regarded as interior points. By \cite[Lemma 6.18]{GT01}, the regularity of $\phi_{0,\varepsilon}$ can be improved to $C^{1}(\overline{\Omega})\cap C^{2}\big(\overline{\Omega}\setminus(\{P_3\}\cup\{P_4\})\big)$. In addition, by the maximum principle and Hopf lemma, we get $\phi_{0,\varepsilon}\geq \sqrt{1+|\bm{\xi}|^2}+\varepsilon$ in the $\overline{\Omega}\setminus\overline{\Gamma_{cone}^{\infty}}$. Hence, $0\in J_{\varepsilon}$.

\textit{Step} 2. $J_{\varepsilon}$ is closed. Consider a sequence $(\mu_{m},\phi_{m,\varepsilon})_{m\in\mathbb{N}}$ where $\mu_{m} \in J_{\varepsilon}$ and $\phi_{m,\varepsilon}$ denotes the solution of the corresponding problem. Assume that $\mu_{m} \to \mu_{\infty}$. We will apply the standard regularity theory to prove the closeness of $J_{\varepsilon}$. 

To improve the regularity of $\phi_{m,\varepsilon}$, we utilize the linear property of \eqref{eq: for Fmu}. Define 
\begin{multline}\label{eq: for phi form}
  \sum^{2}_{i,j=1} a_{ij}(\mu; \bm{\xi}, \phi, D\phi) \partial_{ij}\phi:
=c^{2}(\Delta \phi+D^{2} \phi[\bm{\xi},\bm{\xi}])-\mu D^{2}\phi[D\phi-\chi\bm{\xi},D\phi-\chi\bm{\xi}],
\end{multline}
where $\partial_{ij}\phi:=\partial_{\xi_i\xi_j}\phi$ with $i,j\in\{1,2\}$. Since the functions $\phi_{m,\varepsilon}$ satisfy estimates \eqref{eq1:3.17} and \eqref{eq:3.42}, the linearized equation
\begin{equation}\label{eq: for phi linear}
\sum^{2}_{i,j=1} a_{ij}(\mu_{m}; \bm{\xi}, \phi_{m,\varepsilon}, D\phi_{m,\varepsilon}) \partial_{ij}\phi=0
\end{equation}
is uniformly elliptic. Then, for any compact subdomain $\Omega^{c}_{sub}\subset \Omega$, the uniform boundedness of $\phi_{m,\varepsilon}$ in $C^\infty(\Omega^{c}_{sub})$, with respect to the parameter $\mu\in[0,1]$, follows from the interior estimates in \cite[Theorem 6.17]{GT01}. Also as mentioned in Lemma \ref{lemma: lip-estimate}, the points on $\Gamma_{py}\cup\Gamma_{sy1}\cup\Gamma_{sy2}$ can be regarded as interior points of $\Omega$. Besides, the formed corners at $P_1,~P_2$ can be eliminated by using reflection. Thus, from \cite[Corollary 6.7 and Lemma 6.18]{GT01}, we obtain $\phi_{m,\varepsilon}\in C^{2,\alpha_1}$ around the Dirichlet boundary $\overline{\Gamma_{cone}^{\infty}}$ with a uniform estimate independent of $\mu$, where $\alpha_1\in (0,1)$. 

For the corner points $P_3, P_4$, we apply \cite[Lemma 1.3]{Lieb88} to their respective neighborhoods to improve the regularity of $\phi_{m,\varepsilon}$. This yields $|\phi_{m,\varepsilon}|^{-1-\alpha_2}_2\leq C$, where the constant $\alpha_2=\alpha_2(\sigma_1,\sigma_2)\in(0,1)$, and the constant $C$ is independent of $\mu$. We refer the reader to \cite{GH80,Lieb88} for the norm $|\cdot|^{-1-\alpha_2}_{2}$. From  $\|\phi_{m,\varepsilon}\|_{c^{1,\alpha_2}}\leq|\phi_{m,\varepsilon}|^{-1-\alpha_2}_2$, it follows that $\phi_{m,\varepsilon}\in C^{1,\alpha_2}$ at the points $P_3,P_4$.

We verify here that \cite[Lemma 1.3]{Lieb88} is applicable to the problem \eqref{eq: for phi linear} and \eqref{eq:nianxing BC}.  When $\mu_{m}\in J_{\varepsilon}$, the coefficients of equation \eqref{eq: for phi linear} belong to $C^{0}(\overline{\Omega})\cap C^1\big(\overline{\Omega}\setminus(\{P_3\}\cup\{P_4\})\big)$. Moreover, the absence of lower-order terms in this equation implies that the conditions (1.4a)-(1.5d) in \cite{Lieb88} are automatically satisfied. The uniform ellipticity of equation \eqref{eq: for phi linear}, combined with the angles at $P_3$ and $P_4$ lying in $(\frac{\pi}{2},\pi)$, guarantees that the remaining conditions (1.6a)-(1.7) in \cite{Lieb88} are likewise satisfied.

In conclusion, $\phi_{m,\varepsilon}$ is uniformly bounded in $C^{1,\alpha_2}(\overline{\Omega})$ with respect to $\mu\in[0,1]$. Then, we can seek a subsequence $\mu_{m_{k}}\subset\mu_{m}$ such that $\phi_{m_{k},\varepsilon}\to\phi_{m_{\infty},\varepsilon}$ in $C^{1}(\overline{\Omega})$. Moreover, for any subdomain $\Omega_{sub}\subset\overline{\Omega}\setminus(\{P_3\}\cup\{P_4\})$, using the Arzela-Ascoli theorem, there must exist a subsequence $\mu_{m_{k'}}\subset\mu_{m_k}$ such that the corresponding functions $\phi_{m_{k'},\varepsilon}$ converge in $C^2(\Omega_{sub})$, which implies $\phi_{m_\infty,\varepsilon}\in C^1(\overline{\Omega})\cap C^2\big(\overline{\Omega}\setminus(\{P_3\}\cup\{P_4\})\big)$. Then, the limit function $\phi_{m_{\infty},\varepsilon}$ is a solution to problem \eqref{eq: for Fmu}--\eqref{eq:nianxing BC} and $\phi_{m_{\infty},\varepsilon}\geq \sqrt{1+|\bm{\xi}|^2}+\varepsilon$ in $\overline{\Omega}\setminus\overline{\Gamma_{cone}^{\infty}}$; that is, $\mu_{\infty}\in J_{\varepsilon}$.

\textit{Step} 3. $J_{\varepsilon}$ is open. Let $\mu_{0}\in J_{\varepsilon}$ and $\phi_{\mu_{0},\varepsilon}$ be the solution of the corresponding problem. We introduce a auxiliary function $s_{\mu,\varepsilon}$, given by 
	\begin{equation}\label{eq: for tran. z}
		{\phi}_{\mu,\varepsilon}(\bm{\xi})=\sqrt{1+|\bm{\xi}|^{2}}\cosh s_{\mu,\varepsilon}(\bm{\xi}).
	\end{equation}
From the above relation and \eqref{eq: for Fmu}--\eqref{eq:nianxing BC}, a tedious but direct computation yields that $s_{\mu,\varepsilon}$ satisfies
\begin{multline}\label{eq: for auxiliary z}
	(1+n(\bm{\xi},Ds))(\Delta s+D^{2}s[\bm{\xi},\bm{\xi}])-\mu(1+|\bm{\xi}|^{2})D^{2}s[Ds+(Ds\cdot\bm{\xi})\bm{\xi},Ds+(Ds\cdot\bm{\xi})\bm{\xi}]\\
	+2\Big(1+(1-\mu)n(\bm{\xi},Ds)\Big)Ds\cdot\bm{\xi}+\Big(2+(1-\mu)n(\bm{\xi},Ds)\Big)\frac{1+n(\bm{\xi},Ds)}{(1+|\bm{\xi}|^{2})\tanh s}=0.
\end{multline}
with
\begin{equation*}\label{eq:b c for auxiliary z}
    \begin{cases}
      s=\text{arcosh}\, (1+\varepsilon) \quad&\text{on $\Gamma_{cone}^{\infty}$},\\ 
      Ds\cdot(\bm{\nu}_{py}-\tan\sigma_2\bm{\xi})=0\quad&\text{on $\Gamma_{py}$},\\
      Ds\cdot\bm{\nu}_{sy1}=0\quad&\text{on $\Gamma_{sy1}$},\\
      Ds\cdot\bm{\nu}_{sy2}=0\quad&\text{on $\Gamma_{sy2}$},
    \end{cases}
\end{equation*}
where
\begin{equation*}
	n(\bm{\xi},Ds):=(1+|\bm{\xi}|^2)(|Ds|^2+|Ds\cdot\bm{\xi}|^2).
\end{equation*}
Note that, in contrast to \eqref{eq: for Fmu}, the principal coefficients of \eqref{eq: for auxiliary z} depend only on $\bm{\xi},Ds$. Moreover, the lower-order terms of this equation is non-increasing with respect to $s$. These structural properties are crucial for the solvability of the corresponding linearized equation \eqref{eq:3.46} below.
    
Equation \eqref{eq: for auxiliary z} can be rearranged into the following form:
\begin{equation}\label{eq1:3.1}
 \sum^{2}_{i,j=1} A_{ij}(\mu;\bm{\xi},Ds)\partial_{ij}s+L(\mu;\bm{\xi},s,Ds)=0,
\end{equation}
Corresponding, the linearization of equation \eqref{eq1:3.1} at $s_{\mu_{0},\varepsilon}$ is
\begin{equation}\label{eq:3.46}
	 \sum^{2}_{i,j=1} A_{ij}(\mu_0;\bm{\xi},Ds_{\mu_{0},\varepsilon})\partial_{ij}s+B_i(\mu_0;\bm{\xi},Ds_{\mu_{0},\varepsilon})\partial_i s+\partial_{s}L(\mu_{0};\bm{\xi},s_{\mu_{0},\varepsilon},Ds_{\mu_{0},\varepsilon})s=f,
\end{equation}
where $s_{\mu_{0},\varepsilon}$ is defined by $\phi_{\mu_{0},\varepsilon}$ and \eqref{eq: for tran. z}, and
\begin{align*}
	B_{i}=\sum^{2}_{i,j=1}\partial_{p_i}A_{ij}(\mu_{0};\bm{\xi},Ds_{\mu_{0},\varepsilon})\partial_{ij}s_{\mu_{0},\varepsilon}+\partial_{p_i}L(\mu_{0};\bm{\xi},s_{\mu_{0},\varepsilon},Ds_{\mu_{0},\varepsilon})\quad\text{$i=1,2$}.
\end{align*}
Here $(p_1,p_2):=(\partial_{\xi_1}s,\partial_{\xi_2}s)$. Due to $\partial_{s}L(\mu_{0};\bm{\xi},s_{\mu_{0},\varepsilon},Ds_{\mu_{0},\varepsilon})\leq0$, equation \eqref{eq:3.46} satisfies the maximum principle. Then, it follows from \cite[Appendix]{CQ19} that the corresponding oblique derivative problem is uniquely solvable. Similar to the argument in \cite{CQ19}, the linearized operator for the nonlinear mapping $\mu\mapsto s_{\mu,\varepsilon}$ is invertible at $(\mu_{0},s_{\mu_{0},\varepsilon})$. By the implicit function theorem, there consequently exists a neighborhood of $\mu_{0}$ where the mapping $\mu\mapsto s_{\mu,\varepsilon}$ is well-defined; namely, $\mu_{0}$ is an interior point of $J_{\varepsilon}$.

In conclusion, the set $J_{\varepsilon}$ is both open and closed with $0\in J_{\varepsilon}$, i.e., $J_{\varepsilon}=[0,1]$, which means the existence of $\phi_{1,\varepsilon}$.

We now establish the existence of solutions to problem \eqref{eq: problem for phi} (i.e., Problem \ref{prob2}). 
From the estimates \eqref{eq1:3.17} and \eqref{eq:3.42}, we see that $\phi_{1,\varepsilon}$ is uniformly bounded in $Lip(\overline{\Omega})$ with respect to the parameter $\varepsilon$. Hence, $\phi_{1,\varepsilon}$ is convergent in the function space $C^0(\overline{\Omega})$, with denoting $\phi$ the limit. Similar to the argument in the \textit{Step} 2, it follows that $\phi\in C^{0}(\overline{\Omega})\cap C^{1}(\overline{\Omega}\setminus\overline{\Gamma_{cone}^{\infty}})\cap C^2\big(\overline{\Omega}\setminus(\overline{\Gamma_{cone}^{\infty}}\cup\{P_3\}\cup\{P_4\})\big)$, which is a solution to problem \eqref{eq: problem for phi}. 

Furthermore, estimates \eqref{eq:strictly bounded} and \eqref{eq:3.42} imply the locally uniform ellipticity of equation \eqref{eq: 2-D eq.}. Then, from \eqref{eq: conical tran.}, \eqref{eq: spherical tran.} and Remark \ref{remark:relax cd},  Lemmas \ref{lemma: inf-estimate} and \ref{lemma: lip-estimate} can be applied to problem \eqref{eq: problem for phi}, yielding $\phi\in Lip(\overline{\Omega})\cap C^{1}(\overline{\Omega}\setminus\overline{\Gamma_{cone}^{\infty}})\cap C^2\big(\overline{\Omega}\setminus(\overline{\Gamma_{cone}^{\infty}}\cup\{P_3\}\cup\{P_4\})\big)$. Using \cite[Theorem 6.17]{GT01} and \cite[Lemma 1.3]{Lieb88}, we further obtain $\phi\in Lip(\overline{\Omega})\cap C^{1,\alpha}(\overline{\Omega}\setminus\overline{\Gamma_{cone}^{\infty}})\cap C^\infty\big(\overline{\Omega}\setminus(\overline{\Gamma_{cone}^{\infty}}\cup\{P_3\}\cup\{P_4\})\big)$.

The proof of our main theorem is complete.

\appendix

\section{The ellipticity of equation \texorpdfstring{\eqref{eq: 2-D eq.}}{(2-D eq.)}}\label{sec:appendix a}
 As discussed in subsection \ref{sec:2.2}, for the Chaplygin gas, equation \eqref{eq: 2-D eq.} is hyperbolic outside the domain $\Omega$, and becomes parabolic degenerate on the boundary $\Gamma_{cone}^{\infty}$. This naturally leads to the question of whether parabolic bubbles exist inside $\Omega$. The question can be answered by the following lemma. Thanks to the rotational invariance of equation \mbox{\eqref{eq:3-D for Phi}}, we can assume that the domain $\Omega$ does not contain the origin $(0,0)$ in conical coordinates.
 \begin{lemma}\label{DeYu-1}
	Let $\Omega_D\subset \mathbb{R}^{2}$ be an open bounded domain that does not contain the origin $(0,0)$. Suppose that a function $\phi\in C^2(\Omega_D)$ satisfies \eqref{eq: 2-D eq.} with $0\leq L\leq1$ in $\Omega_D$. Then,
	\begin{equation*}
	    L^2<1\quad\text{in}~\Omega_D.
	\end{equation*}
\end{lemma}

\begin{proof}
     If the conclusion would not hold, then there exists a point $\hat{P}(\xi_{1p},\xi_{2p})\in \Omega_D$, such that
     \begin{equation}\label{eq:L2=1}
         L^2=1\quad\text{at}~\hat{P}.
     \end{equation}
     Combing with the assumption $0\leq L\leq1$ in $\Omega_D$, we have
     \begin{equation}\label{1fircond}
         D(L^2)(\hat{P})=\bm{0}.
     \end{equation}
     Since equation \eqref{eq:3-D for Phi} is rotation-invariant and the solution $\Phi$ of \eqref{eq:3-D for Phi} takes the form \eqref{eq: conical tran.}, we can assume without loss of generality that $\partial_{x_1}\Phi=\partial_{x_2}\Phi=0$ and $\partial_{x_3}\Phi=|D_{\bm{x}}\Phi|$ at the points $(\xi_{1p}x_3,\xi_{2p}x_3,x_3)$ with $x_3\in \mathbb{R}$. Then it follows from \eqref{eq: conical tran.}  and \eqref{eq: velocity3}--\eqref{eq: 2D Mach num.} that
     \begin{equation}\label{velocm1}
        D\phi=\bm{0} \quad\text{at}~\hat{P}.
    \end{equation}
    Also using \eqref{eq: 2D Mach num.} and \eqref{eq:L2=1}, we obtain 
    \begin{equation}\label{velocm2}
        \phi^2=1+|\bm{\xi}|^{2} \quad\text{at}~\hat{P}.
    \end{equation}
    Moreover, because equation \eqref{eq: 2-D eq.} is also rotation-invariant, without loss of generality, we may choose $\xi_{1p}=-|\bm{\xi}|<0$ and $\xi_{2p}=0$. 
    
    We will arrive at a contradiction for \mbox{\eqref{1fircond}}. From \mbox{\eqref{eq: 2D Mach num.}}, one has
     \begin{equation}\label{eq: DL2}
        \begin{split}
           D(L^2)&=\frac{D(|D\phi|^2+\chi^2)-L^2Dc^2}{c^2}-\frac{(1+|\bm{\xi}|^2)D\phi^2-\phi^2D(|\bm{\xi}|^2)}{(1+|\bm{\xi}|^2)^2c^2}\\
            &=\frac{Dc^2}{c^2}(1-L^2)-\frac{(1+|\bm{\xi}|^2)D\phi^2-\phi^2D(|\bm{\xi}|^2)}{(1+|\bm{\xi}|^2)^2c^2}.
        \end{split}
    \end{equation}
    Here, for the second identity, we have used the relation \mbox{$Dc^2=D(|D\phi|^2+\chi^2)$}. Substituting the relations \mbox{\eqref{eq:L2=1} and \eqref{velocm2}} into \mbox{\eqref{eq: DL2}} gives
    \begin{equation*}
        \begin{split}
            D(L^2)&=-\frac{2(1+|\bm{\xi}|^2)\phi D\phi-2\phi^2\bm{\xi}}{(1+|\bm{\xi}|^2)^2c^2}\quad\text{at}~\hat{P},\\
            &=\frac{2\bm{\xi}}{(1+|\bm{\xi}|^2)c^2}=-\frac{2}{(1+|\xi_{1p}|^2)c^2}(|\xi_{1p}|,0)\quad\text{at}~\hat{P}.
        \end{split}
    \end{equation*}
    Namely, $D_{\xi_1}(L^2)<0$ at $\hat{P}$, which is contrary to \mbox{\eqref{1fircond}}. The proof is complete.
\end{proof}    

\section*{Declarations} 

\textbf{Funding:} This work was supported in part by Natural Science Foundation of Hubei Province of China (Grant number 2024AFB007).

\vspace{2mm}

\textbf{Data Availability Statement:} Data sharing is not applicable to this article because no datasets were generated/analyzed during the preparation of the paper.

\vspace{2mm}

\textbf{Conflict of interest:} The author has no relevant financial or non-financial interests to disclose.

\end{document}